\documentclass[leqno]{siamltex704}
\usepackage{amssymb,amsmath,graphicx,amscd,mathrsfs}
\usepackage{color,xcolor,amsmath}
\usepackage{amsmath}
\usepackage{graphicx}
\usepackage{mathrsfs}
\usepackage{float}
\usepackage{amsfonts,amssymb}
\usepackage{dsfont}
\usepackage{pifont}
\usepackage{hyperref}
\usepackage{multirow}
\usepackage{placeins}
\numberwithin{equation}{section}
\def\3bar{{|\hspace{-.02in}|\hspace{-.02in}|}}
\def\E{{\mathcal{E}}}
\def\T{{\mathcal{T}}}

\def\dQ{{\mathbb{Q}}}

\def\b0{\boldsymbol{0}}
\def\sumT{\sum_{T\in\mathcal{T}_h}}     %new
\def\bn{{\mathbf{n}}}

\def\bf{{\mathbf{f}}}
\def\bq{{\mathbf{q}}}

\newtheorem{remark}{Remark}[section]
\newtheorem{algorithm1}{Weak Galerkin Algorithm}

 \newcommand{\eps}{\varepsilon}

 \newcommand{\Real}{\mathbb{R}}

 \newcommand{\trb}[1]{|\!|\!|#1|\!|\!|}
 \newcommand{\tdu}{\hat u_h}

\allowdisplaybreaks % For long formula

\begin{document}
\title{The Shifted-inverse Power Weak Galerkin Method for Eigenvalue Problems}

\author{
Qilong Zhai\thanks{Department of Mathematics, Jilin University, Changchun, 130012,
China (diql15@mails.jlu.edu.cn).}
\and Xiaozhe Hu\thanks{Department of Mathematics, Tufts University, Medford, 02155, USA (Xiaozhe.Hu@tufts.edu).}
\and Ran Zhang\thanks{Department of Mathematics, Jilin University, Changchun, 130012, China
(zhangran@mail.jlu.edu.cn). The research was supported in part by China Natural National Science Foundation
(U1530116,  91630201),  and by the Program for New Century Excellent Talents in University of Ministry of
Education of China, Program for Cheung Kong Scholars, Key Laboratory of Symbolic Computation and Knowledge Engineering of
Ministry of Education, Jilin University, Changchun, 130012, P.R. China}
}

\maketitle

\begin{abstract}
This paper proposes and analyzes a new weak Galerkin method for the eigenvalue problem by using the shifted-inverse power technique.  A high order lower bound can be obtained at a relatively low cost via the proposed method. The error estimates for both eigenvalue and eigenfunction are provided and asymptotic lower bounds are shown as well under some conditions.  Numerical examples are presented to validate the theoretical analysis.
\end{abstract}



\section{Introduction}
The eigenvalue problems have drawn much attention during the past several decades and have wide applications in physical and industrial fields, such as quantum mechanics, fluid mechanics, stochastic process, structural mechanics. More applications of eigenvalue problems are illustrated in \cite{Grebenkov2013} and the references therein.

Many numerical methods have been developed for solving eigenvalue problems, such as finite difference method \cite{Kuttler1974a,MR967886}, finite element method \cite{MR1115240,Babuska1989}, spectral method \cite{Min2004}, and discontinuous Galerkin method \cite{Cliffe2010}. However, there are still two difficulties in solving eigenvalue problems.  One is that the eigenvalue problem is a fully nonlinear problem and the computational cost is very high. Therefore it is important to design algorithms to reduce the computational complexity. The other difficulty is getting a lower bound of an eigenvalue. Due to the minimum-maximum principle, the conforming finite element approximations always produce upper bound of the exact eigenvalue.  If a lower bound is given, then we can get a interval to which the eigenvalue belongs and derive a more accurate approximate eigenvalue.

Numerical techniques have been developed to accelerate the computation of the eigenvalue problems. A two-grid method was firstly proposed by Xu in \cite{Xu1994} for semi-linear partial differential equations (PDEs).  It was soon been applied to nonlinear PDEs \cite{Xu1996} and the eigenvalue problems \cite{Xu2001}. The main idea of the two-grid method is to solve the eigenvalue problem on a coarse grid and a linear problem on a fine grid, instead of solving the eigenvalue problem on the fine grid directly.  Meanwhile, the asymptotic convergence rate is maintained as long as fine grid mesh size $h$ and coarse grid mesh size $H$ are chosen properly.  For example, for the Laplacian eigenvalue problem, the ratio of mesh sizes of two grids can be $H=\sqrt h$, which shall greatly reduces the computation cost.  The two-grid method has also been used in many other problems \cite{Xu2000,Yang2009}, and some multigrid methods have also been proposed \cite{Chen2016,Ji2013,Xie2014}.

Based on the two-grid method, a shifted-inverse power technique was developed \cite{Hu2011,Yang2011}, which further reduces the computational cost because the coarse grid mesh size can be chosen as $H=\sqrt[4]h$.  The shifted-inverse power technique can also be combined with other numerical methods, such as multigrid method \cite{Chen2015,Emi2013} and adaptive algorithm \cite{Bi2016}, which can solve eigenvalue problems more efficiently.

On the other hand, since the conforming finite element methods fail to produce a lower bound for the eigenvalues naturally, a variety of non-standard finite element methods have been developed.  A posterior analysis was proposed to provide a lower bound \cite{Canc2015,Larson2000}.  Many non-conforming elements have also been studied for the lower bound problem, such as Wilson's element, $EQ_1^{rot}$ element, and GCR element\cite{Lin2008}.  Some criterions for non-conforming elements have been studied in \cite{Hu2011a,Hu2014b,Hu2014c} and some numerical methods of getting both upper and lower bounds have been discussed in \cite{Hu2014d}.

Among the numerous methods above, the weak Galerkin (WG) method is also a candidate for solving the lower bound problem. The weak Galerkin finite element method was proposed by Wang and Ye in \cite{Wang2013a} and can be applied on polytopal/polyhedra mesh. The key of weak Galerkin method is to employ discontinuous basis functions and use specifically defined weak derivatives to replace the classical derivatives.  The weak Galerkin method has been applied to many types of PDEs, such as biharmonic equation \cite{Oden2007,Mu2013,Zhang2015}, Stokes equation \cite{Wang2015b,Zhai2015}, Brinkman equation \cite{Mu,Wang2016a,Zhai2016}, and Maxwell equation \cite{Mu2013d}. In \cite{Xie2015}, the weak Galerkin method has been used to solve the Laplacian eigenvalue problems and provide asymptotic lower bounds of arbitrary high order.

In this paper, we combine the shifted-inverse power technique with the weak Galerkin method.  The shifted-inverse power technique reduces the computational cost of weak Galerkin method, while the weak Galerkin method provides a lower bound estimate under certain conditions.  Therefore, by combining the weak Galerkin method with the shifted-inverse power method, we are able to get a high order lower bound efficiently.

This paper is constructed as follows. In Section 2, the weak Galerkin scheme in the general setting is introduced.  Section 3 is devoted to the error analysis for the shifted-inverse power weak Galerkin method.  In Section 4, the application of the proposed method to Laplacian and biharmonic eigenvalue problems are analyzed.  Numerical experiments are presented in Section 5.  %In Section 6, the detailed analysis for the biharmonic eigenvalue problem is given.

\section{A weak Galerkin scheme}
In this section, we introduce the weak Galerkin scheme for the eigenvalue problem \eqref{problem-eq1} and the weak Galerkin scheme based on the shifted-inverse power technique.

We first introduce some notations and definitions.  Suppose $\{V,(\cdot,\cdot)_a\}$ is a Hilbert space and $(\cdot,\cdot)_b$ is another inner-product on $V$.  Let $W$ be the completion of $V$ with respect to $(\cdot,\cdot)_b$, then $\{W,(\cdot,\cdot)_b\}$ is also a Hilbert space.  Assume $V$ is compact embedded into $W$.  Denote $a(w,v) = (w,v)_a$ and $b(p,q)=(p,q)_b$, then $a(\cdot,\cdot)$ and $b(\cdot,\cdot)$ are
symmetric bilinear forms on $V$ and $W$, respectively.

For any $u\in W$, by the Riesz representation theorem, there exists a unique $Au\in V$ such that
\begin{eqnarray*}
  a(Au,v) = b(u,v),\quad\forall v\in V,
\end{eqnarray*}
which define a linear compact operator $A: W\rightarrow W$.
Similarly, we define a bounded linear operator $L: V\rightarrow W$ satisfying
\begin{eqnarray*}
  a(u,v) = b(Lu,v),\quad\forall v\in V.
\end{eqnarray*}

For a Banach space $X$ and its closed subspaces $M$ and $N$, define
the distances as follows,
\begin{equation*}
dist(x,N)=\inf_{y\in N}\|x-y\|, \ dist(M,N)=\sup_{\substack{x\in M,\\\|x\|=1}} dist(x,N),
%\\
%&&\hat\delta(M,N)=\max\{\delta(M,N),\delta(N,M)\}.
\end{equation*}

%Since $A$ is a symmetric coercive compact operator, then the eigenvalues of $A$ is
%a positive sequence tends to zero. Denote $\{(\mu_i,w_i)\}_{i=1}^\infty$ is the
%eigenpairs of $A$, and $\lambda_i = 1/\mu_i$, then we have
%\begin{eqnarray*}
%  L w_i = \lambda_i w_i,
%\end{eqnarray*}
%i.e. $\{(\lambda_i,w_i)\}$ is the eigenpair of $L$ with $\lambda_i$ tends to infinity
%as $i$ tends to infinity.

%In this paper, we are concerned with the eigenvalue problem
Based on those definitions, we consider the following eigenvalue problem
\begin{eqnarray*}
  L u =\lambda u,
\end{eqnarray*}
which can also be written as the following the variational form: Find $u\in V$, $\lambda\in\Real$, such that $b(u,u)=1$ and
\begin{eqnarray}\label{problem-eq1}
  a(u,v) = \lambda b(u,v),\quad\forall v\in V.
\end{eqnarray}

Now, we introduce the weak Galerkin method for problem (\ref{problem-eq1}).  Define $V_h$ the weak Galerkin finite element space.  Note that $V_h$ consists discontinuous piecewise polynomials and is not a subspace of $V$.  Denote $Q_h$ the projection operator from $V$ onto $V_h$. \textcolor{red}{(Q:  is this the projection with respect to the $b$ bilinear form? I think we need to specify it.)}  Let $a_w(\cdot,\cdot)$ and $b_w(\cdot,\cdot)$ be two bilinear forms on $V_h$ and $\trb v^2 =a_w(v,v)$ defines a norm on $V_h$ and $\|v\|^2 = b_w(v,v)$ defines a semi-norm on $V_h$.  Then, $a_w(\cdot,\cdot)$ is bounded and coercive.  The original weak Galerkin algorithm for eigenvalue problem (\ref{problem-eq1}) is as follows,
\begin{algorithm1}\label{wg-alg}
Find $(\lambda_h,u_h)\in \Real\times V_h$, such that
$b_w(u_h,u_h)=1$ and
\begin{eqnarray}\label{wg-scheme}
a_w(u_h,v) = \lambda_h b_w(u_h,v),\quad\forall v\in V_h.
\end{eqnarray}
\end{algorithm1}

For the $i$-th eigenvalue $\lambda$ of problem \eqref{problem-eq1} with multiplicity $N_i$,  we denote the corresponding eigenfunction by $\{u_{j}\}_{j=1}^{N_i}$.  The corresponding WG approximation are denoted by $\{(\lambda_{j,h},u_{j,h})\}_{j=1}^{N_i}$.  Let $M = \text{span}\{ u_{1}, u_2, \cdots, u_{N_i} \}$ be the eigenspace of $\lambda_i$ and $M_h = \text{span}\{ u_{1,h}, u_{2,h}, \cdots, u_{N_i,h} \}$ be the corresponding WG approximation.
%$\{(\lambda_{j,h},u_{j,h})\}_{j=1}^{N_i}$ the WG approximation from scheme (\ref{wg-scheme}),
%and $M_h$ the eigenspace of $\{u_{j,h}\}_{j=1}^{N_i}$.
Define
\begin{eqnarray}
&&\delta_h=\max_{1\le j\le N_i}|\lambda-\lambda_{j,h}|, \label{def:delta}
\\
&&\sigma_h=\min_{1\le j\le N_i}|\lambda-\lambda_{j,h}|, \label{def:sigma}
\\
&&\eta_h=\max_{1\le j\le N_i} \min_{u\in M}\|u_{j,h}-Q_h u\|, \label{def:eta}
\\
&&\gamma_h=\max_{1\le j\le N_i} \min_{u\in M}\trb{u_{j,h}-Q_hu}. \label{def:gamma}
\end{eqnarray}
Note, in the rest, we might replace the subscript $h$ by $H$ when those quantities are defined on $H$.

Next we introduce the weak Galerkin algorithm based on
the shifted-inverse power technique. The algorithm is illustrated
as follows.
\begin{algorithm1}\label{shift-alg}
Step 1. Solve an eigenvalue problem on coarse grid:  Find $(\lambda_H,u_H)\in \Real\times V_H$, such that
$b_w(u_H,u_H)=1$ and
\begin{eqnarray}\label{wg-step1}
a_w(u_H,v_H) = \lambda_H b_w(u_H,v_H),\quad\forall v_H\in V_H.
\end{eqnarray}

Step 2. Solve a linear system on fine grid:  Find $\tilde u_h\in V_h$ such that
\begin{eqnarray}\label{wg-step2}
a_w(\tilde u_h,v_h)-\lambda_H b_w(\tilde u_h,v_h) = b_w(u_H,v_h),\quad\forall v_h\in V_h.
\end{eqnarray}

Step 3. Calculate the Rayleigh quotient,
\begin{eqnarray}\label{wg-step3}
\tilde\lambda_h=\dfrac{a_w(\tilde u_h,\tilde u_h)}{b_w(\tilde u_h,\tilde u_h)}.
\end{eqnarray}
\end{algorithm1}

\section{Error analysis}
In this section, we shall establish the convergence analysis for Algorithm \ref{shift-alg}.  Moreover,  under certain conditions, the lower bound estimate of the approximate eigenvalues is also derived.

Assume $\{(\lambda_{j,h}, u_{j,h})\}_{j=1}^{N_i}$ are the approximations, corresponding to the $i$-th eigenvalue $\lambda$ of the eigenvalue problem \eqref{problem-eq1} with multiplicity $N_i$, obtained by the WG scheme \eqref{wg-scheme}.  Therefore, we have,
\begin{eqnarray*}
a_w(u_{j,h},v_h) = \lambda_{j,h} b_w(u_{j,h},v_h),\quad\forall v_h\in V_h.
\end{eqnarray*}
%Notice that $\{(\lambda_{j,h}, u_{j,h})\}_{j=1}^{N_i}$ is an analysis tool in this section and does not need to be calculated actually.

Next we introduce several technique tools for the error estimate.  Firstly, the following lemma plays an essential role in the convergence analysis and its detailed proof can be found in \cite{Hu2011}, Lemma 1.
%\begin{lemma}\label{eig-identity}
%Suppose $(\lambda,u)\in \Real\times V$ satisfies
%\begin{eqnarray*}
%a(u,v)=\lambda b(u,v),\quad\forall v\in V.
%\end{eqnarray*}
%Then, for any $w\in V$ and $b(w,w)\ne 0$, we have
%\begin{eqnarray*}
%\frac{a(w,w)}{b(w,w)}-\lambda=\frac{a(w-u,w-u)}{b(w,w)}-\lambda\frac{b(w-u,w-u)}{b(w,w)}.
%\end{eqnarray*}
%\end{lemma}
\begin{lemma}\label{technique-1}
Assume $\mu\neq\lambda_{j,h}$ for all $j=1,\cdots,N_i$, when $h$ is small enough we have
\begin{eqnarray*}
a_w(v,v)-\mu b_w(v,v)\ge C_\rho \trb{v}^2,\quad\forall v\in M_h^\perp,
\end{eqnarray*}
where $C_\rho$ only depends on $\lambda$ and $M_h^\perp$ is the
orthogonal complement space of $M_h$ in $V_h$.
\end{lemma}

Secondly, we have the following discrete Poincare inequality holds true on $V_h$, which has been proved in \cite{Li2013}, Lemma 4.2.
\begin{lemma}\label{dPoincare}
For any $v_h\in V_h$, we have
\begin{eqnarray*}
\|v_h\|\lesssim \trb{v_h}.
\end{eqnarray*}
\end{lemma}

Based on those lemmas, we are ready to derive the following main convergence theory of the WG scheme based on the shifted-inverse power techniques, i.e., Algorithm \ref{shift-alg}.
\begin{theorem}\label{Error_Theorem_Shift}
Suppose $\lambda$ is the $i$-th eigenvalue of problem (\ref{problem-eq1}) with multiplicity $N_i$,  $(\tilde\lambda_h, \tilde u_h)$ is the approximate eigenpair obtained by Algorithm \ref{shift-alg} and $\{(\lambda_{j,h}, u_{j,h})\}_{j=1}^{N_i}$ is the approximate eigenpair obtained by Algorithm \ref{wg-alg}.  Assume $\lambda_H$ is not the eigenvalue of Algorithm \ref{wg-alg}. And when $H$ and $h$ are sufficiently small, $\delta_h \leq \frac{1}{4} \sigma_H$ and $dist(u_H,M_h)\gtrsim \eta_H$, then  there exists $\lambda_h^E\in [\lambda_h^{min},\lambda_h^{max}]$ satisfying
\begin{eqnarray*}
|\lambda_h^E-\tilde\lambda_h|\lesssim \sigma_H^{-2}\delta_H^{4}\eta_H^2,
\end{eqnarray*}
where $\lambda_h^{min}=\min\{\lambda_{j,h}\}_{j=1}^{N_i}$ and $\lambda_h^{max}=\max\{\lambda_{j,h}\}_{j=1}^{N_i}$.  Furthermore, suppose $\bar u_h=\tilde u_h/\|\tilde u_h\|$, then there exists an
exact eigenfunction $u$, such that
\begin{eqnarray*}
|\lambda-\tilde\lambda_h|&\lesssim& \sigma_H^{-2}\delta_H^{4}\eta_H^2+\delta_h,
\\
\trb{Q_hu-\bar u_h}&\lesssim& \sigma_H^{-1}\delta_H^2\eta_H+\sigma_H^{-1}\delta_H\gamma_h.
\end{eqnarray*}
\end{theorem}
\begin{proof}
%Suppose $(\tilde\lambda_h, \tilde u_h)$ is the numerical solution of Algorithm \ref{shift-alg}. Without loss of generality, suppose
From \eqref{wg-step2}, we have
\begin{eqnarray*}
  a_w(\tilde u_h,v_h)-\lambda_{m,H} b_w(\tilde u_h,v_h) = b_w(u_{m,H},v_h),\quad\forall v_h\in V_h,
\end{eqnarray*}
where $(\lambda_{m,H}, u_{m,H})$ is the $m$-th numerical eigenpair corresponding to $\lambda$ on the coarse grid.  Define $\tdu\in V_h$ such that
\begin{eqnarray}\label{est1}
a_w(\tdu,v_h)-\lambda_{m,H} b_w(\tdu,v_h) = (\lambda_{m,h}-\lambda_{m,H})b_w(u_{m,H},v_h),\quad\forall v_h\in V_h.
\end{eqnarray}
Then $\tdu=1/(\lambda_{m,h}-\lambda_{m,H})\tilde u_h$ and they have the same Rayleigh quotient, i.e.
\begin{eqnarray*}
\tilde\lambda_h=\frac{a_w(\tdu,\tdu)}{b_w(\tdu,\tdu)}.
\end{eqnarray*}

Let $E_h\tdu$ be the orthogonal projection of $\tdu$ on $M_h$ with respect to $b_w(\cdot, \cdot)$, then $\tdu-E_h\tdu\in M_h^\perp$, $E_h\tdu\in M_h$, and
\begin{eqnarray}\label{est7}
  \|\tdu\|^2=\|E_h\tdu\|^2+\|\tdu-E_h\tdu\|^2.
\end{eqnarray}
Denote $E_h\tdu=\sum_{j=1}^{N_i}\alpha_j u_{j,h}$ and then we have
\begin{equation*}
a_w(\tdu-E_h\tdu,E_h\tdu) = \sum_{j=1}^{N_i}\alpha_j a_w(\tdu-E_h\tdu,u_{j,h}) =\sum_{j=1}^{N_i}\lambda_{j,h}\alpha_j b_w(\tdu-E_h\tdu,u_{j,h})=0
\end{equation*}
%\begin{eqnarray*}
%  &&a_w(\tdu-E_h\tdu,E_h\tdu)
%  \\
%  &=&\sum_{j=1}^{N_i}\alpha_j a_w(\tdu-E_h\tdu,u_{j,h})
%  \\
%  &=&\sum_{j=1}^{N_i}\lambda_{j,h}\alpha_j b_w(\tdu-E_h\tdu,u_{j,h})=0,
%\end{eqnarray*}
which implies $E_h \tdu$ also the orthogonal projection with respect to $a_w(\cdot, \cdot)$ and
\begin{eqnarray}\label{est8}
  \trb\tdu^2=\trb{E_h\tdu}^2+\trb{\tdu-E_h\tdu}^2.
\end{eqnarray}

%Denote $T_h:V_h\rightarrow V_h$ the solver operator defined by $b_w(A_h v_h,w_h)=a_w(v_h,w_h)$ for any $w_h\in V_h$, and denote $A_h$ its inverse.
Define operator $A_h$ by $b_w(A_h v_h,w_h)=a_w(v_h,w_h)$ for any $w_h\in V_h$, then (\ref{est1}) can be rewritten in the following operator form,
\begin{eqnarray}\label{est5}
(A_h-\lambda_{m,H} I)\hat u_h=(\lambda_{m,h}-\lambda_{m,H})u_{m,H}.
\end{eqnarray}
Assume $\lambda_{m,H}$ is not an eigenvalue of $A_h$, then $(A_h-\lambda_{m,H} I)$ is an
isomorphism from $M_h$ to $M_h$ and from $M_h^\perp$ to $M_h^\perp$. Then we have
\begin{align*}
E_h(A_h-\lambda_{m,H} I)\tdu &= E_h(A_h-\lambda_{m,H} I)E_h\tdu+E_h(A_h-\lambda_{m,H} I)(\tdu-E_h\tdu) \\
&= (A_h-\lambda_{m,H} I)E_h\tdu
\end{align*}
%\begin{eqnarray*}
%&&E_h(A_h-\lambda_{m,H} I)\tdu
%\\
%&=&E_h(A_h-\lambda_{m,H} I)E_h\tdu+E_h(A_h-\lambda_{m,H} I)(\tdu-E_h\tdu)
%\\
%&=&(A_h-\lambda_{m,H} I)E_h\tdu,
%\end{eqnarray*}
and it follows that
\begin{eqnarray}\label{est6}
(A_h-\lambda_{m,H} I)E_h\hat u_h=(\lambda_{m,h}-\lambda_{m,H})E_h u_{m,H}.
\end{eqnarray}
From \eqref{est5} and \eqref{est6},  we can conclude that
\begin{eqnarray}\label{est11}
  (A_h-\lambda_{m,H} I)(\hat u_h-E_h\hat u_h)=(\lambda_{m,h}-\lambda_{m,H})(u_{m,H}-E_h u_{m,H}).
\end{eqnarray}

%Suppose $M$ is the exact eigenspace corresponding to $\lambda$. Then we can obtain
%\begin{eqnarray*}
%  dist(u_{m,H},M)&\lesssim& \eta_H,
%  \\
%  dist(M,M_h) &\lesssim& \eta_h.
%\end{eqnarray*}
%Then, it follows that
On the other hand, based on the definition \eqref{def:eta}, we have
\begin{equation} \label{est12}
\|u_{m,H}-E_h u_{m,H}\|= dist(u_{m,H},M_h) \le dist(u_{m,H},M)+dist(M,M_h) \lesssim \eta_H+\eta_h.
\end{equation}
%\begin{eqnarray}
%  \|u_{m,H}-E_h u_{m,H}\|&=&dist(u_{m,H},M_h)
%  \\ \nonumber
%  &\le& dist(u_{m,H},M)+dist(M,M_h)
%  \\ \nonumber
%  &\lesssim& \eta_H+\eta_h.
%\end{eqnarray}
Note that $\hat v=\tdu-E_h\tdu\in M_h^\perp\subset V_h$. From Lemma \ref{technique-1}, \eqref{est11}, and \eqref{est12} we have
\begin{eqnarray*}
\trb{\hat v}^2&\lesssim&a_w(\hat v,\hat v)-\lambda_H b_w(\hat v,\hat v) = b_w((A_h-\lambda_{m,H} I)(\tdu-E_h\tdu),\hat v)
\\
&=&(\lambda_{m,h}-\lambda_{m,H})b_w(u_{m,H}-E_h u_{m,H},\hat v) \lesssim (\delta_h+\delta_H)(\eta_h+\eta_H)\trb{\hat v},
\end{eqnarray*}
which implies
\begin{eqnarray}\label{est3}
\trb{\tdu-E_h\tdu}\lesssim \delta_H\eta_H.
\end{eqnarray}
From (\ref{est3}) and Lemma \ref{dPoincare}, we obtain
\begin{eqnarray*}
  \|\tdu-E_h\tdu\|\lesssim \delta_H\eta_H.
\end{eqnarray*}
And by the boundedness of $A_h$, we have
\begin{eqnarray}\label{est9}
  && \|(A_h-\lambda_{m,H})(\tdu-E_h\tdu)\|
  \lesssim \delta_H\eta_H.
\end{eqnarray}

Since $E_h\tdu=\sum_{j=1}^{N_i}\alpha_ju_{j,h}$, we have
\begin{eqnarray*} %\label{est10}
  \|(A_h-\lambda_{m,H})E_h\tdu\|
  &=&\|\sum_{j=1}^{N_i}\alpha_j (A_h-\lambda_{m,H})u_{j,h}\|
  =\|\sum_{j=1}^{N_i}\alpha_j (\lambda_{j,h}-\lambda_{m,H})u_{j,h}\|
  \\ \nonumber
  &\le& \max|\lambda_{j,h}-\lambda_{m,H}|\|E_h\tdu\|
  \lesssim \delta_H\|\tdu\|.
\end{eqnarray*}
%From  (\ref{est5}), (\ref{est9}), and (\ref{est10}), we can conclude that
%\begin{eqnarray*}
%  1 &=& \|u_{m,H}\|
%  \\
%  &=&\|(\lambda_{m,h}-\lambda_{m,H})^{-1}(A_h-\lambda_{m,H})\tdu\|
%  \\
%  &\lesssim& {\color{red}(\sigma_H-\delta_h)^{-1}}(\|(A_h-\lambda_{m,H})E_h\tdu\|+\|(A_h-\lambda_{m,H})(\tdu-E_h\tdu)\|)
%  \\
%  &\lesssim& \sigma_H^{-1}(\delta_H\eta_H+\delta_H\|\tdu\|),
%\end{eqnarray*}
%which implies
%\begin{eqnarray*}
%  \|\tdu\|\ge C\sigma_H\delta_H^{-1}(1-{\color{red}C\sigma_H^{-1}\delta_H\eta_H}).
%\end{eqnarray*}
Moreover, becasue $(A_h-\lambda_{m,H} I)$ is an
isomorphism from $M_h$ to $M_h$ and from $M_h^\perp$ to $M_h^\perp$,
$(A_h-\lambda_{m,H} I)^{-1}$ is also an
isomorphism from $M_h$ to $M_h$ and from $M_h^\perp$ to $M_h^\perp$.
From (\ref{est5}) we have
\begin{eqnarray*}
  \|\tdu\|^2 &=& \|(\lambda_{m,h}-\lambda_{m,H})(A_h - \lambda_{m,H})^{-1}u_{m,H}\|^2
  \\
  &=& \|(\lambda_{m,h}-\lambda_{m,H})(A_h - \lambda_{m,H})^{-1}E_h u_{m,H}\|^2
  \\
  &&+\|(\lambda_{m,h}-\lambda_{m,H})(A_h - \lambda_{m,H})^{-1}(u_{m,H}-E_h u_{m,H})\|^2
  \\
  &\ge& \|(\lambda_{m,h}-\lambda_{m,H})(A_h - \lambda_{m,H})^{-1}E_h u_{m,H}\|^2
  \\
  &=& \|(\lambda_{m,h}-\lambda_{m,H})(A_h - \lambda_{m,H})^{-1}\left(\sum_{j=1}^{N_i} (u_{j,h},u_{m,H})u_{j,h}\right)\|^2
  \\
  &=& \|\sum_{j=1}^{N_i} \dfrac{\lambda_{m,h}-\lambda_{m,H}}{\lambda_{j,h}-\lambda_{m,H}} (u_{j,h},u_{m,H})u_{j,h}\|^2
  \\
  &\ge& \left(1-\max_{1\le j\le N_i} \left|\dfrac{\lambda_{m,h}-\lambda_{j,h}}{\lambda_{j,h}-\lambda_{m,H}}\right|\right)^2 (\|u_{m,H}\|^2-\|u_{m,H}-E_h u_{m,H}\|^2).
%\\
%&\ge& \frac12(1-C\eta_H^2).
\end{eqnarray*}
%\begin{eqnarray*}
%&&\|(\lambda_{m,h}-\lambda_{m,H})(A_h - \lambda_{m,H})^{-1}E_h u_{m,H}\|^2
%\\
%&=& \|(\lambda_{m,h}-\lambda_{m,H})(A_h - \lambda_{m,H})^{-1}\left(\sum_{j=1}^{N_i} (u_{j,h},u_{m,H})u_{j,h}\right)\|^2
%\\
%&=& \|\sum_{j=1}^{N_i} \dfrac{\lambda_{m,h}-\lambda_{m,H}}{\lambda_{j,h}-\lambda_{m,H}} (u_{j,h},u_{m,H})u_{j,h}\|^2
%\\
%&\ge& \left(1-\max_{1\le j\le N_i} \left|\dfrac{\lambda_{m,h}-\lambda_{j,h}}{\lambda_{j,h}-\lambda_{m,H}}\right|\right)^2 (\|u_{m,H}\|^2-\|u_{m,H}-E_h u_{m,H}\|^2)
%\\
%&\ge& \frac12(1-C\eta_H^2).
%\end{eqnarray*}
%Since $\{u_{j,h}\}_{j=1}^{N_i}$ is an orthogonal basis of $M_h$, we have
%\begin{eqnarray*}
%  E_h u_{m,H} = \sum_{j=1}^{N_i} (u_{j,h},u_{m,H})u_{j,h}.
%\end{eqnarray*}
By the definitions \eqref{def:delta} and \eqref{def:sigma}, the assumption that, when $h$ and $H$ are sufficiently small, $\delta_h  \leq \frac{1}{4} \sigma_H$ , we have
\begin{eqnarray*}
&&\max_{1\le j\le N_i} \left|\frac{\lambda_{m,h}-\lambda_{j,h}}{\lambda_{j,h}-\lambda_{m,H}}\right|
\lesssim  \dfrac{2 \displaystyle\max_{1\le j\le N_i}|\lambda-\lambda_{j,h}|}{\displaystyle|\lambda-\lambda_{m,H}|-\max_{1\le j\le N_i}|\lambda-\lambda_{j,h}|}
 \lesssim  \frac{2\delta_h}{|\sigma_H-\delta_h|} \leq \frac{2}{3}.
\end{eqnarray*}
%It follows from the assumption $\delta_h \ll \sigma_H$ that when $h$ and $H$ are sufficiently small,
%$$\max_{1\le j\le N_i} \left|\frac{\lambda_{m,h}-\lambda_{j,h}}{\lambda_{j,h}-\lambda_{m,H}}\right|\le \frac12,$$
%which yields
Thus, we obtain
\begin{eqnarray}\label{est4}
  \|\tdu\|\ge C.
\end{eqnarray}

{\color{black}
Next, we estimate the eigenfunctions. Since $dist(M_h,M)\lesssim \gamma_h$ \eqref{def:gamma} with respect to
$\trb\cdot$ norm, there exists $u_j\in M$ such that
\begin{eqnarray*}
\trb{E_h\tdu-Q_h u} &=& \trb{\sum_{j=1}^{N_i}\alpha_j(Q_hu_j-u_{j,h})}\lesssim \gamma_h.
\end{eqnarray*}
From (\ref{est3}) we can derive that
\begin{eqnarray*}
  \trb{\tdu-Q_h u}\le \trb{\tdu-E_h\tdu}+\trb{E_h\tdu-Q_h u}\lesssim \delta_H\eta_H+\gamma_h.
\end{eqnarray*}
By the definitions of $\bar u_h$ and $\tdu$, together with the lower bound \eqref{est4}, we define $\bar u=u/\|\tdu\|$ and obtain
\begin{eqnarray*}
  \trb{\bar u_h-Q_h\bar u}\lesssim \trb{\tdu-Q_h u}\lesssim \delta_H\eta_H+\gamma_h.
\end{eqnarray*}
}

Now we turn to the estimate for the eigenvalue. Define
\begin{eqnarray*}
  \lambda_h^E=\dfrac{a_w(E_h\tdu,E_h\tdu)}{b_w(E_h\tdu,E_h\tdu)}.
\end{eqnarray*}
From \eqref{est7}), \eqref{est8}, \eqref{est3}, and \eqref{est4}, we have
\begin{eqnarray*}
|\tilde \lambda_h-\lambda_h^E|
&=&\left|\frac{a_w(\tdu,\tdu)}{b_w(\tdu,\tdu)}-\lambda_h^E\right|
\lesssim \left|a_w(\tdu,\tdu)-\lambda_h^Eb_w(\tdu,\tdu)\right|\|\tdu\|^{-2}
\\
&\lesssim& \left|a_w(\tdu-E_h\tdu,\tdu-E_h\tdu)-\lambda_h^Eb_w(\tdu-E_h\tdu,\tdu-E_h\tdu)\right|
\\
&&+\left|a_w(E_h\tdu,E_h\tdu)-\lambda_h^Eb_w(E_h\tdu,E_h\tdu)\right| \lesssim \delta_H^{2}\eta_H^2.
\end{eqnarray*}
Next we show that $\lambda_h^{min}\le\lambda_h^E\le\lambda_h^{max}$.
Recall that $E_h\tdu=\sum_{j=1}^{N_i}\alpha_j u_{j,h}$, and $\{u_{j,h}\}_{j=1}^{N_i}$ is
an orthonormal basis of $M_h$, we have
\begin{eqnarray*}
  \alpha_j=b_w(E_h\tdu,u_{j,h}).
\end{eqnarray*}
It follows that
\begin{eqnarray*}
  \lambda_h^E %&=& \dfrac{a_w(E_h\tdu,E_h\tdu)}{b_w(E_h\tdu,E_h\tdu)}
 =\dfrac{\sum_{j=1}^{N_i}\alpha_ja_w(E_h\tdu,u_{j,h})}{\sum_{j=1}^{N_i}\alpha_jb_w(E_h\tdu,u_{j,h})}
  =\dfrac{\sum_{j=1}^{N_i}\lambda_{j,h}\alpha_jb_w(E_h\tdu,u_{j,h})}{\sum_{j=1}^{N_i}\alpha_jb_w(E_h\tdu,u_{j,h})}
  =\dfrac{\sum_{j=1}^{N_i}\alpha_j^2\lambda_{j,h}}{\sum_{j=1}^{N_i}\alpha_j^2},
\end{eqnarray*}
which implies $\lambda_h^{min}\le\lambda_h^E\le\lambda_h^{max}$.
Thus, the proof is completed.
\end{proof}

Based on Theorem~\ref{Error_Theorem_Shift}, we have the following corollary shows that $\tilde{\lambda}_h$ is a lower bound of $\lambda$.

\begin{corollary}\label{Shift-lowerbound}
Under the same assumptions of Theorem \ref{Error_Theorem_Shift}, we have
\begin{eqnarray*}
|\tilde\lambda_h-\lambda_h^E|\lesssim \delta_H^{2}\eta_H^2.
\end{eqnarray*}
If $\lambda_h^E$ is a lower bound of $\lambda$, i.e.,
%\begin{eqnarray*}
 $\lambda-\lambda_h^E\ge \sigma_h$,
%\end{eqnarray*}
then $\tilde \lambda_h$ is still a lower bound of $\lambda$ provided $\delta_H^{2}\eta_H^2\ll \sigma_h$. \textcolor{red}{(This assumption does not make sense.  Because usually the shifted-inverse power method achieve the optimal asymptotical error when $\delta_H^2 \eta_H^2 \approx \delta_h$.  And for WG scheme, $\delta_h \geq \sigma_h$ asymptotically!)}
\end{corollary}

\section{Examples}
In this section, we use the Laplacian and biharmonic eigenvalue problems as examples to illustrate the shifted-inverse power weak Galerkin method, i.e., Algorithm~\ref{shift-alg}.

\subsection{Laplacian eigenvalue problem}
Consider the Laplacian eigenvalue problem
\begin{eqnarray}\label{Laplacian-eq1}
-\Delta u &=& \lambda u,\quad\text{ in }\Omega,
\\ \label{Laplacian-eq2}
u&=& 0,\quad\text{ on }\partial\Omega,
\\ \label{Laplacian-eq3}
\int_\Omega u^2 &=& 1,
\end{eqnarray}
where $\Omega$ is a polygon or polyhedral domain in $\Real^d(d=2,3)$.

Let $\T_h$ be a polygonal partition of the domain $\Omega$ satisfying the assumptions in \cite{Wang2014a} and $\E_h$ denote all the edges (faces in 3D) in $\T_h$. We use $P_k(T)$ to represent the piecewise polynomials of degree $k$ on each element $T\in\T_h$ and use $P_k(e)$ to represent the piecewise
polynomials of degree $k$ on each edge $e\in\E_h$.  For each element $T$, $h_T$ stands for the diameter of $T$ and $h=\max_{T\in \T_h} h_T$ is the mesh size.

We introduce the following weak Galerkin finite element space
\begin{eqnarray*}
V_h=\{(v_0,v_b):v_0\in P_k(T), v_b\in P_{k-1}(e),
\text{ and } v_b=0 \text{ on }\partial\Omega\},
\end{eqnarray*}
where $k\ge 1$ is an integer. We emphasis that $v_b$ is single-valued on each $e\in \E_h$ and $v_b$ is irrelevant to the trace of $v_0$.  Now, we define some projections onto $V_h$.  Denote $Q_0$ the $L^2$ projection onto $P_k(T)$ on each element $T$, $Q_b$ the $L^2$ projection onto $P_{k-1}(e)$ on each element $e$, and $Q_h=\{Q_0,Q_b\}$ is a projection operator onto $V_h$.  Moreover,  on the WG space $V_h$, we can define the following weak gradient operator $\nabla_w$ by distribution.
\begin{definition}
For any $v\in V_h$, $\nabla_w v\in
[P_{k-1}(T)]^d$ is the unique polynomial
satisfying on each element $T\in \T_h$,
\begin{eqnarray*}
(\nabla_w v,\bq)_T = -(v_0,\nabla\cdot\bq)_T + \langle
v_b,\bq\cdot\bn\rangle_{\partial T},
\quad\forall\bq\in [P_{k-1}(T)]^d,
\end{eqnarray*}
where $\bn$ is the unit outward normal vector.
\end{definition}

%Now we introduce the following three bilinear forms on $V_h$.
Based on the above definitions and notations, the bilinear forms $a_w(\cdot, \cdot)$ and $b_w(\cdot, \cdot)$ for Laplacian eigenvalue problem are defined as follows, for any $v,w\in V_h$,
\begin{align*}
a_w(v,w)&=(\nabla_w v,\nabla_w w)+s(v,w), \\
b_w(v,w)=&(v_0,w_0),
\end{align*}
where
\begin{equation*}
s(v,w)=\sumT h_T^{-1+\eps}\langle Q_b v_0-v_b,
Q_b w_0-w_b\rangle_{\partial T}
\end{equation*}
and $0<\eps<1$ is a positive constant to be chosen.
Furthermore, we define a semi-norm
\begin{eqnarray*}
\trb{v}^2=a_w(v,v),
\end{eqnarray*}
which indeed defines a norm on $V_h$ as shown in \cite{Xie2015}.

%With these preparations, we can introduce the following
%original Weak Galerkin
%Algorithm \ref{wg-scheme} for Laplacian eigenvalue problem (\ref{problem-eq1}).
%\begin{algorithm1}
%Find $(\lambda_h,u_h)\in \Real\times V_h$, such that
%$b_w(u_h,u_h)=1$ and
%\begin{eqnarray}
%a_w(u_h,v) = \lambda_h b_w(u_h,v),\quad\forall v\in V_h.
%\end{eqnarray}
%\end{algorithm1}

For the WG scheme for Laplacian eigenvalue problems, i.e., Algorithm~\ref{wg-scheme}, the following convergence results has been derived in Theorem 4.7 and 5.3, \cite{Xie2015}.
\begin{lemma}\label{Error_Theorem_Eigenpair}
Let $\lambda_{i,h}$ be the $i$-th approximate eigenvalue obtained by Algorithm \ref{wg-alg} and $u_{i,h}$ be the corresponding eigenvector.  There exists an exact eigenvalue $\lambda_i$ and the corresponding exact eigenfunction $u_i$ such that, if $u_i\in H^{k+1}(\Omega)\cap H_0^1(\Omega)$, the following error estimates hold
\begin{eqnarray}  \nonumber %\label{eig-est1}
&&h^{2k}\lesssim \lambda_i-\lambda_{i,h}\lesssim Ch^{2k-2\eps},
\\\label{eig-est2}
&&\trb{Q_hu_i-u_{i,h}} \lesssim h^{k-\eps},
\\\label{eig-est3}
&&\|u_i-u_{i,h}\|\lesssim h^{k+1-\eps}.
\end{eqnarray}
\end{lemma}
By Lemma \ref{Error_Theorem_Eigenpair}, we have
\begin{equation*}
\delta_h\lesssim h^{2k-2\eps}, \sigma_h \gtrsim h^{2k}, \eta_h\lesssim h^{k+1-\eps}, \ \text{and} \ \gamma_h\lesssim h^{k-\eps}.
\end{equation*}
%\begin{eqnarray*}
%&&\delta_h\lesssim h^{2k-2\eps},
%\\
%&&\sigma_h\ge Ch^{2k},
%\\
%&&\eta_h\lesssim h^{k-\eps},
%\\
%&&\gamma_h\lesssim h^{k+1-\eps}.
%\end{eqnarray*}

Suppose ($\tilde\lambda_h$, $\tilde u_h$) is the approximate eigenpair obtained by shifted-inverse
power weak Galerkin algorithm \ref{shift-alg}. Let $\bar u_h=\tilde u_h/\|\tilde u_h\|$.  According to Theorem \ref{Error_Theorem_Shift},  when $h^{2k-2\eps} \leq \frac{1}{4} H^{2k}$, we have
\begin{eqnarray*}
|\lambda-\tilde\lambda_h|&\lesssim& \delta_H^{2}\eta_H^2+\delta_h\lesssim H^{6k+2-6\eps}+h^{2k-2\eps},
\\
\trb{Q_hu-\bar u_h}&\lesssim& \delta_H\eta_H+\gamma_h\lesssim H^{3k+1-3\eps}+h^{k-\eps}.
\end{eqnarray*}
Moreover, since WG approximation $\lambda_h$ is a lower bound of $\lambda$, by Corollary \ref{Shift-lowerbound}, $\tilde\lambda_h$ is still a lower bound of $\lambda$ if $H^{6k+2-6\eps}\ll h^{2k}$. \textcolor{red}{(Again, this does not make sense. The optimal error suggest us to choose $H^{3k+1-3\eps} = h^{k-\eps}$, plug back in and we obtain $h^{2k-2\eps} \ll h^{2k}$, which simply is wrong!)}

\begin{remark}
  The assumptions in Theorem \ref{Error_Theorem_Shift} requires $h^{2k-2\eps} \leq \frac{1}{4} H^{2k}$ which can be easily satisfied.  For example,  the error estimate suggests to choose $H^{3k+1-3\eps} = h^{k-\eps}$.  Therefore, the requirement becomes $H^{6k+2-6\eps} \leq \frac{1}{4} H^{2k}$ which holds when $H$ is sufficiently small.
%meet when $h$ is sufficiently small.
%When $k=1$ and $\eps=0.1$, this condition reduces to $h \ll H^{1.11}$, which is satisfied in the numerical examples in Section 5.
\end{remark}

\subsection{Biharmonic eigenvalue problem}
Consider the biharmonic eigenvalue problem
\begin{eqnarray}\label{biharmonic-eq1}
\Delta^2 u &=& \lambda u,\quad\text{ in }\Omega,
\\ \label{biharmonic-eq2}
u= \frac{\partial u}{\partial n}&=& 0,\quad\text{ on }\partial\Omega,
\\ \label{biharmonic-eq3}
\int_\Omega u^2 &=& 1,
\end{eqnarray}
where $\Omega$ is a polygon or polyhedral domain in $\Real^d(d=2,3)$.

For biharmonic problem, the weak Galerkin finite element space is defined as follows,
\begin{eqnarray*}
V_h=\{(v_0,v_b,v_n):v_0\in P_k(T), v_b\in P_{k-1}(e), v_n\in P_{k-1}(e),
\text{ and } v_b=v_n=0 \text{ on }\partial\Omega\},
\end{eqnarray*}
where $k\ge 2$ is an integer.  We define some projections onto $V_h$ as usual. Denote $Q_0$ the
$L^2$ projection onto $P_k(T)$ on each element $T$, $Q_b$ the
$L^2$ projection onto $P_{k-1}(e)$ on each element $e$, and
$Q_h v=\{Q_0 v,Q_b v,Q_b(\nabla v\cdot\bn_e)\}$ is a projection operator onto $V_h$. Moreover, on the finite element space $V_h$, we define the weak Laplacian operator $\Delta_w$ by distribution as follows,
\begin{definition}
For any $v\in V_h$, $\Delta_w v\in
P_{k-2}(T)$ is the unique polynomial
satisfying on each element $T\in \T_h$,
\begin{eqnarray*}
(\Delta_w v,\varphi)_T = (v_0,\Delta\varphi)_T - \langle
v_b,\nabla\varphi\cdot\bn\rangle_{\partial T}
+\langle v_n(\bn_e\cdot\bn),\varphi\rangle_{\partial T},
\quad\forall\varphi\in P_{k-2}(T),
\end{eqnarray*}
where $\bn$ is the unit outward normal vector and $\bn_e$ is the unit normal vector on each edge.
\end{definition}

Now we introduce the bilinear forms on $V_h$.  For any $v,w\in V_h$, define
\begin{eqnarray*}
a_w(v,w)&=&(\Delta_w v,\Delta_w w)+s(v,w),
\\
b_w(v,w)&=&(v_0,w_0),
\end{eqnarray*}
where
\begin{align*}
s(v,w)&=\sumT h_T^{-3+\eps}\langle Q_b v_0-v_b,
Q_b w_0-w_b\rangle_{\partial T}  \\
& \quad +\sumT h_T^{-1+\eps}\langle \nabla v_0\cdot\bn_e-v_n,
\nabla w_0\bn_e-w_n\rangle_{\partial T},
\end{align*}
and $0<\eps<1$ is a positive constant to be chosen.
Furthermore, define
\begin{eqnarray*}
\trb{v}^2=a_w(v,v).
\end{eqnarray*}
And according to Lemma \ref{biharmonic-norm}, $\trb\cdot$ indeed defines
a norm on $V_h$.

%With these preparations, we can introduce the following
%original Weak Galerkin
%Algorithm \ref{wg-scheme} for biharmonic eigenvalue problem (\ref{biharmonic-eq1})
%-(\ref{biharmonic-eq3}).
%\begin{algorithm1}\label{biharmonic-alg}
%Find $(\lambda_h,u_h)\in \Real\times V_h$, such that
%$b_w(u_h,u_h)=1$ and
%\begin{eqnarray}
%a_w(u_h,v) = \lambda_h b_w(u_h,v),\quad\forall v\in V_h.
%\end{eqnarray}
%\end{algorithm1}

For the weak Galerkin scheme for biharmonic eigenvalue problem, the following convergence theorem holds true as shown in Theorem \ref{Error_Theorem_Biharmonic} and \ref{biharmonic-lowerbound} in Appendix \ref{appendix}.
\begin{lemma}\label{Error_Lemma_Biharmonic}
Let $\lambda_{i,h}$ be the $i$-th approximate eigenvalue obtained by Algorithm \ref{wg-scheme} and $u_{i,h}$ be the corresponding eigenvector.  There exists an exact eigenvalue $\lambda_i$ and the corresponding exact eigenfunction $u_i$ such that, if $u_i\in H^{k+2}(\Omega)\cap H_0^2(\Omega)$, the following error estimates hold
\begin{eqnarray}
&&h^{2k-2}\lesssim \lambda_i-\lambda_{i,h}\lesssim Ch^{2k-2-2\eps},
\\
&&\trb{Q_hu_i-u_{i,h}} \lesssim h^{k-1-\eps},
\\
&&\|u_i-u_{i,h}\|\lesssim h^{k+k_0-2-\eps},
\end{eqnarray}
where $k_0=\min\{k,3\}$.
\end{lemma}

According to Lemma \ref{Error_Lemma_Biharmonic}, we have
\begin{equation*}
\delta_h\lesssim h^{2k-2-2\eps}, \sigma_h\gtrsim Ch^{2k-2}, \eta_h\lesssim h^{k+k_0-2-\eps}, \ \text{and} \ \gamma_h\lesssim h^{k-1-\eps}.
\end{equation*}
%\begin{eqnarray*}
%&&\delta_h\lesssim h^{2k-2-2\eps},
%\\
%&&\sigma_h\ge Ch^{2k-2},
%\\
%&&\eta_h\lesssim h^{k+k_0-2-\eps},
%\\
%&&\gamma_h\lesssim h^{k-1-\eps}.
%\end{eqnarray*}
Let ($\tilde\lambda_h$, $\tilde u_h$) be the approximate eigenpair of the shifted-inverse
power weak Galerkin Algorithm \ref{shift-alg} and $\bar u_h=\tilde u_h/\|\tilde u_h\|$.
From Theorem \ref{Error_Theorem_Shift}, When $h^{2k-2-2\eps} \leq \frac{1}{4} H^{2k-2}$, we have
\begin{eqnarray*}
|\lambda-\tilde\lambda_h|&\lesssim& \delta_H^{2}\eta_H^2+\delta_h\lesssim H^{6k+2k_0-8-6\eps}+h^{2k-2-2\eps},
\\
\trb{Q_hu-\bar u_h}&\lesssim& \delta_H\eta_H+\gamma_h\lesssim H^{3k+k_0-4-3\eps}+h^{k-1-\eps}.
\end{eqnarray*}
Moreover, since $\lambda_h$ is a lower bound of $\lambda$, from Corollary \ref{Shift-lowerbound}
it follows that when $H^{6k+2k_0-8-10\eps}\ll h^{2k-2}$, $\tilde\lambda$ is still a lower bound of $\lambda$. \textcolor{red}{(Again, double check this!)}
%{\color{black}
%\begin{remark}
%  The assumptions in Theorem \ref{Error_Theorem_Shift} requires $\sigma_H \ll \delta_h$, i.e. $h^{2k-2-2\eps}\ll H^{2k-2}$.
%  When $k=1$ and $\eps=0.1$, this condition reduces to $h \ll H^{1.11}$, which is satisfied in the numerical examples in Section 5.
%\end{remark}
%}

\section{Numerical Experiments}
In this section, we present some numerical results to show the efficiency of the shifted-inverse power weak Galerkin method and verify the theoretical analysis in the previous sections.
\subsection{Example 1}
Consider the Laplacian eigenvalue problem \eqref{Laplacian-eq1}-\eqref{Laplacian-eq3} on a unit
square domain $(0,1)\times(0,1)$.
The exact eigenvalues are
\begin{eqnarray*}
  \lambda=(m^2+n^2)\pi^2,
\end{eqnarray*}
and the corresponding eigenfunctions are
\begin{eqnarray*}
  u=\sin(\pi mx)\sin(\pi ny),
\end{eqnarray*}
where $m$, $n$ are positive integers.  We solve the problem (\ref{problem-eq1}) by Algorithm \ref{shift-alg}. The uniform mesh is employed and the parameter $\eps$ is set to be $0.1$. For the case $k=1$, the error of first six eigenvalues and eigenfunctions are listed in Tables \ref{Tables_Exam1_1} and \ref{Tables_Exam1_2}.
%All the eigenfunctions are normalized.
When $H \leq 1/8$ and fix $h = 1/512$,  the coarse mesh size is sufficiently small and the error is dominated by the term related to $h$.  It should also be noticed that all the numerical eigenvalues are lower bounds,
which coincides with the theoretical prediction.

\begin{table}[!htp]
\centering \caption{The errors for the eigenvalue approximations $\lambda-\tilde\lambda_h$
for Example 1.}\label{Tables_Exam1_1}
\begin{tabular}{|c|c|c|c|c|c|c|c|c|c|c|c|c|c|}
\hline
$H$   &    1/8   &    1/16  &    1/32  &    1/64  \\ \hline
$h$   &    1/512 &    1/512 &    1/512 &    1/512 \\ \hline
$\lambda_1-\tilde \lambda_{1,h}$ & 5.9031e-4 & 5.9045e-4 & 5.9045e-4 & 5.9045e-4 \\ \hline
$\lambda_2-\tilde \lambda_{2,h}$ & 3.7769e-3 & 3.8294e-3 & 3.8295e-3 & 3.8295e-3 \\ \hline
$\lambda_3-\tilde \lambda_{3,h}$ & 3.7805e-3 & 3.8294e-3 & 3.8295e-3 & 3.8295e-3 \\ \hline
$\lambda_4-\tilde \lambda_{4,h}$ & 9.0348e-3 & 9.4457e-3 & 9.4464e-3 & 9.4464e-3 \\ \hline
$\lambda_5-\tilde \lambda_{5,h}$ & 1.2305e-2 & 1.5744e-2 & 1.5750e-2 & 1.5750e-2 \\ \hline
$\lambda_6-\tilde \lambda_{6,h}$ & 1.2293e-2 & 1.5744e-2 & 1.5750e-2 & 1.5750e-2 \\ \hline
\end{tabular}
\end{table}

\begin{table}[!htp]
\centering \caption{The errors for the eigenfunction approximations $\trb{Q_h u-\tilde u_h}$
for Example 1.}\label{Tables_Exam1_2}
\begin{tabular}{|c|c|c|c|c|c|c|c|c|c|c|c|c|c|}
\hline
$H$   &    1/8   &    1/16  &    1/32  &    1/64  \\ \hline
$h$   &    1/512 &    1/512 &    1/512 &    1/512 \\ \hline
$\trb{Q_h u_1-\tilde u_{1,h}}$ & 2.2642e-2 & 2.2639e-2 & 2.2639e-2 & 2.2639e-2 \\ \hline
$\trb{Q_h u_2-\tilde u_{2,h}}$ & 5.8240e-2 & 5.7795e-2 & 5.7795e-2 & 5.7795e-2 \\ \hline
$\trb{Q_h u_3-\tilde u_{3,h}}$ & 5.8240e-2 & 5.7795e-2 & 5.7795e-2 & 5.7795e-2 \\ \hline
$\trb{Q_h u_4-\tilde u_{4,h}}$ & 9.2885e-2 & 9.0558e-2 & 9.0554e-2 & 9.0554e-2 \\ \hline
$\trb{Q_h u_5-\tilde u_{5,h}}$ & 1.3216e-1 & 1.1744e-1 & 1.1742e-1 & 1.1742e-1 \\ \hline
$\trb{Q_h u_6-\tilde u_{6,h}}$ & 1.3216e-1 & 1.1744e-1 & 1.1742e-1 & 1.1742e-1 \\ \hline
\end{tabular}
\end{table}

{\color{black}
\subsection{Example 2}
Consider the biharmonic eigenvalue problem (\ref{biharmonic-eq1})-(\ref{biharmonic-eq3}) on a unit
square domain $(0,1)\times(0,1)$.
The first eigenvalue is $\lambda = 1.2949339598e + 003$.
We solve the problem (\ref{problem-eq1}) by Algorithm
\ref{shift-alg}. The uniform mesh is employed and the
parameter $\eps$ is set to be $0.1$. For the case $k=3$, the error
of the first eigenvalue are listed in Tables \ref{Tables_Exam1_3}.
%All the eigenfunctions are normalized.
From Tables \ref{Tables_Exam1_3}, we can see that $H \leq 1/16$, the coarse mesh size is sufficiently small and the error is dominated by the term related to $h$.  It should also be noticed that all the numerical eigenvalues are lower bounds,  which coincides with the theoretical prediction.

\begin{table}[!htp]
\centering \caption{The errors for the first eigenvalue
for Example 2.}\label{Tables_Exam1_3}
\begin{tabular}{|c|c|c|c|c|c|c|c|c|c|c|c|c|c|}
\hline
$H$   &    1/8   &    1/16  &    1/32  &    1/64  \\ \hline
$h$   &    1/128 &    1/128 &    1/128 &    1/128 \\ \hline
$\lambda$ & \multicolumn{4}{|c|}{1294.933959} \\ \hline
$\tilde\lambda_h$ & 1294.925382 & 1294.914299 & 1294.914289 & 1294.914289 \\ \hline
$\lambda -\tilde \lambda_{h}$ & 8.5772e-3 & 1.9660e-2 & 1.9669e-2 & 1.9670e-2\\ \hline \hline
$H$   &    1/8   &    1/16  &    1/32  &    1/64  \\ \hline
$h$   &    1/256 &    1/256 &    1/256 &    1/256 \\ \hline
$\lambda$ & \multicolumn{4}{|c|}{1294.933959} \\ \hline
$\tilde\lambda_h$ & 1294.943654 & 1294.932580 & 1294.932570 & 1294.932569\\ \hline
$\lambda -\tilde \lambda_{h}$ & -9.6946e-3 & 1.3795e-3 & 1.3897e-3 & 1.3904e-3\\ \hline
\end{tabular}
\end{table}
}

\appendix
\section{Error Analysis for the Biharmonic Eigenvalue Problem} \label{appendix}
In this section, we shall give the error analysis for Algorithm
\ref{wg-scheme} solving the biharmonic eigenvalue problem (\ref{biharmonic-eq1})-(\ref{biharmonic-eq3}).
Both the error of eigenvalues and eigenfunctions are analyzed and the lower bound estimate is also given.

\subsection{Preliminaries}
%Consider the following Weak Galerkin
%Algorithm \ref{wg-scheme} for problem (\ref{biharmonic-eq1})-(\ref{biharmonic-eq3}).
%It is necessary to introduce some norms on the finite element space $V_h$.
Define a semi-norm on $V_h$ that for any $v\in V_h$ as follows,
\begin{eqnarray*}
\trb{v}^2=a_w(v,v).
\end{eqnarray*}
Next lemma show that it actually is a norm.
\begin{lemma}\label{biharmonic-norm}
  $\trb{\cdot}$ defines a norm on $V_h$.
\end{lemma}
\begin{proof}
Notice that if $\trb{v}=0$, then on each element $T\in\T_h$
we have $\Delta_w v=0$ in $T$,
$Q_b v_0=v_b$ and $\nabla v_0\cdot\bn_e=v_n$ on $\partial T$, which implies
\begin{eqnarray*}
&&(\Delta v_0,\Delta v_0)
\\
&=&(v_0,\Delta(\Delta v_0))-\sumT\langle v_0,\nabla(\Delta v_0)
\cdot\bn\rangle_{\partial T}+\sumT\langle \nabla v_0\cdot\bn,\Delta v_0
\rangle_{\partial T}
\\
&=&(v_0,\Delta(\Delta v_0))-\sumT\langle v_b,\nabla(\Delta v_0)
\cdot\bn\rangle_{\partial T}+\sumT\langle \nabla v_n\bn_e\cdot\bn,\Delta v_0
\rangle_{\partial T}
\\
&&-\sumT\langle Q_b v_0-v_b,\nabla(\Delta v_0)
\cdot\bn\rangle_{\partial T}+\sumT\langle (\nabla v_0\cdot\bn_e-v_n)\bn_e\cdot\bn,\Delta v_0
\rangle_{\partial T}
\\
&=&(\Delta_w v,\Delta v_0)=0.
\end{eqnarray*}
Then we know that $v_0\in H_0^2(\Omega)$ and $\Delta v_0=0$.
From the uniqueness of the Poisson equation, it follows that
$v=0$, which means $\trb{\cdot}$ is a norm on $V_h$.
\end{proof}

For the convenience of analysis, we introduce another semi-norm on $V_h$. For any $v\in V_h$,
define
\begin{eqnarray*}
\trb{v}_1^2&=&(\Delta_w v,\Delta_w v)+\sumT h_T^{-3}\langle Q_b v_0-v_b,
Q_b v_0-v_b\rangle_{\partial T}
\\
&&+\sumT h_T^{-1}\langle \nabla v_0\cdot\bn_e-v_n,
\nabla v_0\cdot\bn_e-v_n\rangle_{\partial T}.
\end{eqnarray*}
Similar to $\trb{\cdot}$, $\trb{\cdot}_1$ also defines
a norm on $V_h$. Obviously, the two norms have the following
relationship
\begin{eqnarray*}
\trb{v}\lesssim \trb{v}_1 \lesssim h^{-\frac\eps 2}\trb{v},\quad\forall
v\in V_h.
\end{eqnarray*}
Denote $V_0=H_0^2(\Omega)$ and $V=V_0+V_h$. For any
$v\in V$, define
\begin{eqnarray*}
\|v\|_V^2&=&(\Delta v,\Delta v)+\sumT h_T^{-3}\langle Q_b v_0-v_b,
Q_b v_0-v_b\rangle_{\partial T}
\\
&&+\sumT h_T^{-1}\langle \nabla v_0\cdot\bn_e-v_n,
\nabla v_0\cdot\bn_e-v_n\rangle_{\partial T}.
\end{eqnarray*}
For $v\in V_0$, $v_b$ stands for the trace of $v$ on $\partial T$ and $v_n$ stands for $\nabla v\cdot\bn_e$. In the following lemmas we show that $\|\cdot \|_v$ defines a norm on $V$, which is equivalent to $H^2$ norm on $V_0$ and $\trb{\cdot}_1$ norm on $V_h$.  And the proofs are similar to Lemma 4.1 and Lemma 4.2 in \cite{Xie2015} which, therefore, will be omitted here.
\begin{lemma}
For any $v\in V_0$,
\begin{eqnarray*}
\|v\|_V \simeq \|v\|_2.
\end{eqnarray*}
\end{lemma}

\begin{lemma}
For any $v\in V_h$,
\begin{eqnarray*}
\|v\|_V \simeq \trb{v}_1.
\end{eqnarray*}
\end{lemma}

Consider the following biharmonic equation
\begin{eqnarray}\label{BVP-eq1}
-\Delta^2 u &=& f,\quad \text{ in }\Omega,
\\ \label{BVP-eq2}
u=\dfrac{\partial u}{\partial n}&=& 0,\quad\text{ on }\partial\Omega.
\end{eqnarray}
The following weak Galerkin scheme can be established for
problem (\ref{BVP-eq1})-(\ref{BVP-eq2}).
\begin{algorithm1}
Find $u_h\in V_h$, such that
\begin{eqnarray}\label{wgbvp-scheme}
a_w(u_h,v) = (f,v_0),\quad\forall v\in V_h.
\end{eqnarray}
\end{algorithm1}
\textcolor{red}{(What is the $a_w(\cdot, \cdot)$) here?}

The weak Laplacian operator $\Delta_w$ is a good approximation
of $\Delta$ and the following community property holds true.
The proof can be found in Theorem 3.1 in \cite{Zhang2015}
\begin{lemma}\label{commu-prop}
Suppose $v\in H^2(\Omega)$, then the following equality
holds
\begin{eqnarray*}
\Delta_w Q_hv=\dQ_h\Delta v. \text{\textcolor{red}{(Please define $\dQ_h$)}}
\end{eqnarray*}
\end{lemma}
For the biharmonic equation \eqref{BVP-eq1} and \eqref{BVP-eq2}, the following error estimates can be obtained and the detailed proof can be found in \cite{Zhang2015}, Theorem 4.2 and Theorem 5.1.
\begin{theorem}\label{BVP-errest}
Let $u\in H_0^2(\Omega)\cap H^{k+2}(\Omega)$ be the solution of \eqref{BVP-eq1} and \eqref{BVP-eq2} and $u_h$ be the numerical solution of the weak Galerkin scheme \eqref{wgbvp-scheme}, the following estimates
hold
\begin{eqnarray*}
\trb{Q_h u-u_h} &\lesssim& h^{k-1-\frac\eps 2},
\\
\|u-u_h\|_V &\lesssim& h^{k-1-\eps}.
\end{eqnarray*}
\end{theorem}

\begin{theorem}
Let $u\in H_0^2(\Omega)\cap H^{k+2}(\Omega)$ be the solution of
\eqref{BVP-eq1} and \eqref{BVP-eq2}, $u_h$ be the numerical solution of the weak Galerkin scheme \eqref{wgbvp-scheme}. Assume that the dual problem of \eqref{BVP-eq1} and \eqref{BVP-eq2} has $H^4$
regularity, then the following estimates hold,
\begin{eqnarray*}
\|u-u_h\| &\lesssim& h^{k+k_0-2-\eps},
\end{eqnarray*}
where $k_0=\min\{3,k\}$.
\end{theorem}

\subsection{Error estimates for the eigenvalue problem}
In this section, we derive the estimates for the eigenpair of the problem \eqref{biharmonic-eq1}-\eqref{biharmonic-eq3}.  Let $K:L^2(\Omega)\rightarrow H_0^2(\Omega)$ be the solution operator of
the Biharmonic problem \eqref{BVP-eq1} and \eqref{BVP-eq2} and $K_h:L^2(\Omega)\rightarrow
V_h$ be the weak Galerkin numerical solution operator.  Naturally, we can extend the operators $K$ and $K_h$ from $L^2(\Omega)$ to $V$.
\begin{lemma}\label{operator-approx}
The operators $K$ and $K_h$ have the following estimate
\begin{eqnarray*}
\lim_{h\rightarrow 0}\|K_h-K\|_V=0,
\end{eqnarray*}
where $\|\cdot\|_V$ denote the operator norm from $V$ to $V$. \textcolor{red}{(Is this norm different from the one before?  The notation is the same though...)}
\end{lemma}
\begin{proof}
Since $V$ is a Hilbert space, it is equivalent to verify that
\begin{eqnarray*}
\lim_{h\rightarrow 0}\sup_{\|f\|_V=1}\|K f-K_h f\|_V=0.
\end{eqnarray*}
For any $f\in V$ with $\|f\|_V=1$, let $u=Kf$ and $u_h=K_hf$.
From Theorem \ref{BVP-errest} and the regularity of the
problem \eqref{BVP-eq1} and \eqref{BVP-eq2}, we have
\begin{eqnarray*}
\|u- u_h\|_V\lesssim h^{1-\eps}\|u\|_4 \lesssim h^{1-\eps}\|f\|.
\end{eqnarray*}
For $0\le \eps<1$, we have
\begin{eqnarray*}
\|K f- K_h f\|_V\lesssim h^{1-\eps} \|f\|_V,
\end{eqnarray*}
which converges to zero as $h \rightarrow 0$. Thus, the
proof is completed.
\end{proof}

%From the classical PDE theory we know that $T$ is compact on $V$.
\begin{lemma}
The operator $T_h: V\mapsto V$ is compact. \text{\textcolor{red}{(Please define $T_h$})}
\end{lemma}
\begin{proof}
Denote $\tilde T_h$ the restriction of $T_h$ on $V_h$. Since $V_h$ is finite dimensional, $\tilde T_h$ is compact. Notice that $(Q_0 f, v_0)=(f,v_0)$, so $T_h=\tilde T_h Q_h$. In order to prove that $T_h$ is compact, we just need to verify that $Q_h$ is bounded.  For any $w\in V_h$, $Q_h w=w$. For $w\in V_0$, by the equivalence of $\|\cdot\|_V$ and $\|\cdot\|_2$, we can conclude that
\begin{eqnarray*}
\|Q_h w\|_V \le \|Q_h w-w\|_V+\|w\|_V \lesssim \|w\|_2+\|w\|_V
\lesssim C\|w\|_V,
\end{eqnarray*}
which completes the proof.
\end{proof}

Next we review some notations in the spectral approximation theory. We denote by $\sigma(T)$ the spectrum of $T$ and by $\rho(T)$ the resolvent set. $R_z(T)=(z-T)^{-1}$ represents the resolvent operator.  Let $\mu$ be a nonzero eigenvalue of $T$ with algebraic multiplicities $m$. Let $\Gamma$ be a circle in the complex plane which centers at $\mu$, lies in $\rho(T)$, and does not enclose any other eigenvalues in $\sigma(T)$. The corresponding spectral projection is
\begin{eqnarray*}
E=E(\mu)=\frac{1}{2\pi {\rm i}}\int_\Gamma R_z(T) dz.
\end{eqnarray*}
$R(E)$ represents the range of $E$, which is the space of generalized
eigenvectors.  We have the following lemma
%For a Banach space $X$ and its closed subspaces $M$ and $N$, define
%the distances as follows that
%\begin{eqnarray*}
%&&dist(x,N)=\inf_{y\in N}\|x-y\|,~
%\delta(M,N)=\sup_{\substack{x\in M,\\\|x\|=1}} dist(x,N),
%\\
%&&\hat\delta(M,N)=\max\{\delta(M,N),\delta(N,M)\}.
%\end{eqnarray*}
\begin{lemma}
Assume that $w\in H^{k+2}(\Omega)\cap H_0^2(\Omega)$ for any $w\in R(E)$.
Then the following estimate holds true,
\begin{eqnarray*}
\|T-T_h|_{R(E)}\|_V\le Ch^{k-1}.
\end{eqnarray*}
\end{lemma}
\begin{proof}
Suppose $w\in R(E)$ with $\|w\|_V=1$. From Theorem \ref{BVP-errest}
and the regularity of biharmonic equation, we can obtain
\begin{eqnarray*}
\|T w-T_h w\|_V\lesssim h^{k-1-\eps}\|T w\|_{k+2}
\lesssim h^{k-1-\eps}\|w\|_{k+2}.
\end{eqnarray*}
Since $R(E)$ is finite dimensional, there is a uniform upper bound
for $\|w\|_{k+2}$, where $w\in R(E)$ with $\|w\|_V=1$,
which completes the proof.
\end{proof}

Since $a(\cdot,\cdot)$ and $a_w(\cdot,\cdot)$ are symmetric,
$T$ and $T_h$ are self-adjoint. In addition, all the conclusions also hold straightforwardly when replacing the the $\|\cdot\|_V$ norm by the $L^2$-norm.  Therefore, based on the theory in \cite{MR1115240}, we have the following estimates.
\begin{theorem}\label{Error_Theorem_Biharmonic}
Let $\lambda_{j,h}$ be the $j$-th eigenvalue of $T_h$ and $u_{j,h}$ be the corresponding eigenvector.
There exist an exact eigenvalue $\lambda_j$ and the corresponding eigenfunction $u_j$ such that, if $u_j\in H^{k+2}(\Omega)\cap H_0^2(\Omega)$, the following error estimates hold
\begin{eqnarray}\label{biharmonic-est1}
&&|\lambda_j-\lambda_{j,h}|\le Ch^{2k-2-2\eps}\|u_j\|_{k+2},
\\\label{biharmonic-est2}
&&\|u_j-u_{j,h}\|_V \le Ch^{k-1-\eps}\|u_j\|_{k+2},
\\\label{biharmonic-est3}
&&\|u_j-u_{j,h}\|\le Ch^{k+k_0-2-\eps}\|u_j\|_{k+2},
\end{eqnarray}
where $k_0=\min\{k,3\}$.
\end{theorem}

\subsection{Lower bounds}
In this section, we show that the approximate eigenvalue $\lambda_h$ generated by Algorithm \ref{wg-scheme} is a lower bound of $\lambda$ when the parameter $\eps$ is chosen such that $0<\eps<1$.
\begin{lemma}\label{expansion}
Let $(\lambda,u)$ be the eigenpair of \eqref{biharmonic-eq1} and \eqref{biharmonic-eq3} and $(\lambda_h,u_h)$ be the approximate eigenpair obtained by Algorithm \ref{wg-scheme}.  We have the following identity holds for any $v\in V_h$,
\begin{eqnarray}
\nonumber
\lambda-\lambda_h&=&\|\Delta u-\Delta_w u\|^2+s(u_h-v,u_h-v)-\lambda_h
\|u_0-v_0\|^2-\lambda_h(\|u_0\|^2-\|v_0\|^2)
\\   \label{eqn:lam-identity} &&+2(\Delta u-
\Delta_w v,\Delta_w u_h)-s(v,v).
\end{eqnarray}
\end{lemma}
\begin{proof}
From \eqref{biharmonic-eq1}-\eqref{biharmonic-eq3} and Algorithm \ref{wg-scheme}, we have
\begin{eqnarray*}
&&a(u,u)=\|\Delta u\|^2=\lambda\|u\|^2,
\\
&&a_w(u_h,u_h)=\|\Delta_w u_h\|^2+s(u_h,u_h)=\lambda_h\|u_0\|^2,
\\
&&\|u\|=\|u_0\|=1.
\end{eqnarray*}
%Mimicking the expansion in \cite{AD04} and we have the following
%expansion
Note that,
\begin{eqnarray*}
&&(\Delta u-\Delta_w u_h,\Delta u-\Delta_w u_h)
\\
&=&(\Delta u,\Delta u)+(\Delta_w u_h,\Delta_w u_h)-2(\Delta u,
\Delta_w u_h)
\\
&=&\lambda+\lambda_h-2(\Delta u,\Delta_w u_h)-s(u_h,u_h)
\\
&=&\lambda+\lambda_h-2(\Delta u-\Delta_w v,\Delta_w u_h)
-2(\Delta_w v,\Delta_w u_h)-s(u_h,u_h)
\\
&=&\lambda+\lambda_h-2(\Delta u-\Delta_w v,\Delta_w u_h)
-2\lambda_h(u_0,v_0)+2s(u_h,v)-s(u_h,u_h)
\\
&=&\lambda+\lambda_h-2(\Delta u-\Delta_w v,\Delta_w u_h)
+\lambda_h(u_0-v_0,u_0-v_0)-\lambda_h(u_0,u_0)-\lambda_h(v_0,v_0)
\\
&&+2s(u_h,v)-s(u_h,u_h)
\\
&=&\lambda-\lambda_h-2(\Delta u-\Delta_w v,\Delta_w u_h)
+\lambda_h(u_0-v_0,u_0-v_0)+\lambda_h((u_0,u_0)-(v_0,v_0))
\\
&&+2s(u_h,v)-s(u_h,u_h).
\end{eqnarray*}
Rearranging the above identity and \eqref{eqn:lam-identity} follows directly.
%\begin{eqnarray*}
%\lambda-\lambda_h&=&\|\Delta u-\Delta_w u\|^2+s(u_h-v,u_h-v)-\lambda_h
%\|u_0-v_0\|^2-\lambda_h(\|u_0\|^2-\|v_0\|^2)
%\\&&+2(\Delta u-
%\Delta_w v,\Delta_w u_h)-s(v,v),
%\end{eqnarray*}
%which completes the proof.
\end{proof}

Next lemma is crucial for deriving the lower bound of eigenvalues and the detailed proof can be found in \cite{Lin2011a}.
\begin{lemma}\label{convergence-rate}%\rm{(\cite{Lin2011a})}
%Suppose $u\not\in P_k(T)$, $\dQ_h$ is the piecewise $L^2$ projection
%onto $[P_{k-1}(T)]^2$.
Let $u$ be the eigenfunction of the biharmonic eigenvalue problem \eqref{biharmonic-eq1}-\eqref{biharmonic-eq3}, the following lower bound holds
\begin{eqnarray*}
\|\Delta u-\dQ_h \Delta u\|\ge Ch^{2k}.
\end{eqnarray*}
\end{lemma}

Now we are ready to show the lower bound of eigenvalues.
\begin{theorem}\label{biharmonic-lowerbound}
Let $\lambda_j$ and $\lambda_{j,h}$ be the $j$-th exact eigenvalue and its corresponding weak Galerkin numerical approximation.  Assume the corresponding eigenvector $u\in H^{k+2}(\Omega)\cap H_0^2(\Omega)$.  If the mesh size $h$ is small enough, the following estimate holds
\begin{eqnarray*}
\text{\textcolor{red}{$h^{2k-2} \lesssim$ }} \lambda_j-\lambda_{j,h}\lesssim Ch^{2k-2-2\eps}.
\end{eqnarray*}
\end{theorem}
\begin{proof}
Take $v=Q_h u$ in Lemma \ref{expansion}. From the commutative
property in Lemma \ref{commu-prop}, we have
\begin{eqnarray*}
\Delta_w v=\dQ_h \Delta u,
\end{eqnarray*}
and
\begin{eqnarray*}
\lambda-\lambda_h&=&\|\Delta u-\Delta_w u\|^2+s(u_h-v,u_h-v)-\lambda_h
\|u_0-v_0\|^2-\lambda_h(\|u_0\|^2-\|v_0\|^2)
\\&&+2(\Delta u-
\Delta_w v,\Delta_w u_h)-s(v,v)
\\
&=&\|\Delta u-\dQ_h \Delta u\|^2+\trb{Q_hu-u_h}-\lambda_h
\|Q_0 u-u_0\|^2-\lambda_h(\|u_0\|^2-\|Q_0 u\|^2)
\\&&+2(\Delta u-
\dQ_h\Delta u,\Delta_w u_h)-s(Q_h u,Q_h u).
\end{eqnarray*}

Since $\Delta_w u_h\in P_{k-2}(T)$, we can obtain
\begin{eqnarray*}
(\Delta u-\dQ_h\Delta u,\Delta_w u_h)=0.
\end{eqnarray*}
From the error estimate (\ref{eig-est2}) and (\ref{eig-est3}),
we have
\begin{eqnarray*}
\trb{Q_h u-u_h}^2\le\|Q_h u-u_h\|_V^2\lesssim h^{2k-2-2\eps}
\end{eqnarray*}
and
\begin{eqnarray*}
\|Q_0 u-u_0\|^2\lesssim h^{2k+2k_0-4-2\eps},
\end{eqnarray*}
where $k_0=\min\{k,3\}$.  Moreover, we have
\begin{eqnarray*}
\|Q_0 u-u_0\|^2&=&(u_0+Q_0 u,u_0-Q_0 u)
\\
&=&((u-u_0)+(u-Q_0 u),(u-u_0)-(u-Q_0 u))
\\
&=&\|u-u_0\|^2-\|u-Q_0 u\|^2
\\
&\lesssim& h^{2k+2k_0-4-2\eps},
\end{eqnarray*}
and
\begin{eqnarray*}
s(Q_h u,Q_h u) &=&\sumT h_T^{-3+\eps}\langle Q_b Q_0 u-Q_b u,
Q_b Q_0 u-Q_b u\rangle_{\partial T}
\\
&&+\sumT h_T^{-1+\eps}\langle \nabla Q_0 u\cdot\bn_e-Q_b(\nabla
u\cdot \bn_e),
\nabla Q_0 u\cdot\bn_e-Q_b(\nabla
u\cdot \bn_e)\rangle_{\partial T},
\\
&\le& \sumT h_T^{-3+\eps}\|Q_0 u -u\|_{\partial T}^2
+\sumT h_T^{-1+\eps}\|\nabla(Q_0 u -u)\|_{\partial T}^2
\\
&\lesssim& Ch^{2k-2+2\eps}.
\end{eqnarray*}
From Lemma \ref{convergence-rate}, $\lambda_h\|u_0-Q_0 u\|^2$,
$\lambda_h(\|u_0\|^2-\|Q_0 u\|^2)$, $(\Delta u-\dQ_h\Delta u,\Delta_w u_h)$,
and $s(Q_hu, Q_hu)$ are of higher order comparing with $\|\Delta u-\dQ_h\Delta u\|^2$, therefore,
\begin{eqnarray*}
h^{2k-2}\lesssim \trb{Q_h u-u_h}^2+\|\Delta u-\dQ_h\Delta u\|^2
\lesssim h^{2k-2-2\eps}
\end{eqnarray*}
is the dominant term, if $h$ is sufficiently small, which completes the proof.
\end{proof}

\bibliographystyle{siam}
\bibliography{library}

\end{document}